\documentclass{article}
\usepackage[latin1]{inputenc}
\usepackage[T1]{fontenc}
\usepackage{amsmath,amssymb}
\usepackage{amsfonts,amsthm}
\usepackage{geometry}
\usepackage{hyperref}
\usepackage{graphicx}
\usepackage{tikz} %%%% für die Vektorgrafik im Beweis
\usetikzlibrary{decorations.pathreplacing} %%% für die geschweiften Klammern in oben genannter Grafik

\newcommand{\Cb}{\mathbb{C}}

\newcommand{\Nb}{\mathbb{N}}

\newcommand{\Rb}{\mathbb{R}}

\newcommand{\Zb}{\mathbb{Z}}

\newcommand{\Xb}{\textbf{\upshape X}}

\newcommand{\Kc}{\mathcal{K}}
\newcommand{\Lc}{\mathcal{L}}

\newcommand{\Pc}{\mathcal{P}}

\newcommand{\Alp}{\mathcal{A}_{p}}
\newcommand{\Alpr}{\mathcal{A}_{p}^\$}

\newcommand{\ii}{{\bf i}}

\newcommand{\vertiii}[1]{{\left\vert\kern-0.25ex\left\vert\kern-0.25ex\left\vert #1
    \right\vert\kern-0.25ex\right\vert\kern-0.25ex\right\vert}}

\newcommand\pto{
   \unitlength0.1ex
   \begin{picture}(30,15)
   \put(15,16){\makebox(0,0)[]{\tiny $\Pc$}}
   \put(15,5){\makebox(0,0)[]{$\to$}}
   \end{picture}
}

\makeatletter
\newcommand{\xRightarrow}[2][]{\ext@arrow 0359\Rightarrowfill@{#1}{#2}}
\makeatother

\DeclareMathOperator{\diag}{diag}
\DeclareMathOperator{\diam}{diam}
\DeclareMathOperator{\dist}{dist}
\DeclareMathOperator{\im}{im}

\DeclareMathOperator{\op}{op}
\DeclareMathOperator{\ess}{ess}
\DeclareMathOperator{\plim}{\Pc-lim}

\DeclareMathOperator{\plimn}{\underset{\it n\to\infty}{\Pc-lim \,}}
\DeclareMathOperator{\spc}{sp}
\DeclareMathOperator{\spess}{\spc_{\ess}}

\DeclareMathOperator{\supp}{supp}

\providecommand{\lb}[1]{\Lc(#1)}
\providecommand{\lc}[1]{\Kc(#1)}
\providecommand{\pb}[1]{\Lc(#1,\Pc)}
\providecommand{\pbr}[1]{\Lc^\$(#1,\Pc)}
\providecommand{\pc}[1]{\Kc(#1,\Pc)}

\newtheorem{thm}{Theorem}%[chapter]

\newtheorem{prop}[thm]{Proposition}
\newtheorem{cor}[thm]{Corollary}

\theoremstyle{definition}
\newtheorem{defn}[thm]{Definition}
\newtheorem{rem}[thm]{Remark}
\newtheorem{ex}[thm]{Example}

\begin{document}

%%%%%%%%%%%%%%%%%%%%%%%%%%%%%%%%%%%%%%%%%%%%%%%%%%%%%%%%%%%%%%%%%%%%%%%%%%%%%%%
%%%%%%%%%%%%%%%%%%%%%%%%%%%%%%%%%%%%%%%%%%%%%%%%%%%%%%%%%%%%%%%%%%%%%%%%%%%%%%%
%%%%%%%%%%%%%%%%%%%%%%%%%%%%%%%%%%%%%%%%%%%%%%%%%%%%%%%%%%%%%%%%%%%%%%%%%%%%%%%
\title{An Affirmative Answer to a Core Issue on Limit Operators}
\date{}
\author{Marko Lindner\footnote{Technische Universit\"at Hamburg (TUHH), Institut f\"ur Mathematik, 21073 Hamburg, Germany, \url{lindner@tuhh.de}} \, and Markus Seidel\footnote{Technische Universit\"at Chemnitz, Fakult\"at f\"ur Mathematik, 09107 Chemnitz, Germany, \url{markus.seidel@mathematik.tu-chemnitz.de}}}
\maketitle
\vspace{-0.4cm}
\begin{abstract}
An operator on an $l^{p}$-space is called band-dominated if it can be approximated,
in the operator norm, by operators with a banded matrix representation.
It is known that a rich band-dominated operator is $\Pc$-Fredholm (which is a
generalization of the classical Fredholm property) if and only if all of its
so-called limit operators are invertible and their inverses are uniformly bounded.
We show that the condition on uniform boundedness is redundant in this statement.
%which answers a question -- in the positive way -- that has been open for almost 30 years.


\end{abstract}

%%%%%%%%%%%%%%%%%%%%%%%%%%%%%%%%%%%%%%%%%%%%%%%%%%%%%%%%%%%%%%%%%%%%%%%%%%%%%%%
%%%%%%%%%%%%%%%%%%%%%%%%%%%%%%%%%%%%%%%%%%%%%%%%%%%%%%%%%%%%%%%%%%%%%%%%%%%%%%%
%%%%%%%%%%%%%%%%%%%%%%%%%%%%%%%%%%%%%%%%%%%%%%%%%%%%%%%%%%%%%%%%%%%%%%%%%%%%%%%
\section{Introduction}
%%%%%%%%%%%%%%%%%%%%%%%%%%%%%%%%%%%%%%%%%%%%%%%%%%%%%%%%%%%%%%%%%%%%%%%%%%%%%%%
%%%%%%%%%%%%%%%%%%%%%%%%%%%%%%%%%%%%%%%%%%%%%%%%%%%%%%%%%%%%%%%%%%%%%%%%%%%%%%%
%%%%%%%%%%%%%%%%%%%%%%%%%%%%%%%%%%%%%%%%%%%%%%%%%%%%%%%%%%%%%%%%%%%%%%%%%%%%%%%

One of the really great achievements of the limit operator theory is that it provides
powerful tools for the characterization and description of the Fredholm properties
for many important classes of operators, above all the class of band-dominated operators
on generalized $l^p$-sequence spaces.
These results range from the pure characterization of the Fredholmness (or the
generalized $\Pc$-Fredholmness) of an operator $A$ in terms of the invertibility of all
its limit operators \cite{LangeR, RochFredTh, LimOps, Marko, SeFre} to
formulas for the Fredholm index of $A$ \cite{RochFredInd, Rochlp, LpInd, MarkoWiener,
SeSi2, SeSi3}, as well as for the connections of kernel and cokernel dimensions of $A$ to
the respective properties of its finite sections \cite{SeSi2, SeSi3}, formulas
for essential spectra and pseudospectra \cite{BoeTrunc, Standard, RochPS, LimOps, SeSi4}
and even formulas for the essential norm and the essential lower norm \cite{HaLiSe}.

Without going into the details here, one can roughly say that, given an operator $A$,
one obtains its limit operators by considering sequences of shifted copies of $A$ and
by passing to limits of such sequences.
This important family of all limit operators of a given $A$ is usually denoted
by $\sigma_{\op}(A)$ and referred to as its operator spectrum.
The central observation in that machinery is the following
\begin{thm}\label{TUB}
Let $A$ be a rich band-dominated operator. Then $A$ is $\Pc$-Fredholm if and only if
all its limit operators are invertible and their inverses are uniformly bounded.
\end{thm}

\noindent
In \cite{Marko} one of the authors summarizes
\begin{center}
\begin{minipage}[b]{0.9\textwidth}
``It is the biggest question of the whole limit operator business whether or not in
Theorem~1, and consequently in all its applications, uniform invertibility of
$\sigma_{\op}(A)$ can be replaced by elementwise invertibility. In other words:

\medskip

\begin{minipage}[t]{0.05\textwidth}
\quad
\end{minipage}
\begin{minipage}[t]{0.2\textwidth}
\textbf{Big question:}
\end{minipage}
\begin{minipage}[t]{0.7\textwidth}
\textit{Is the operator spectrum of a rich operator automatically
uniformly invertible if it is elementwise invertible?}''

\medskip
\end{minipage}
\end{minipage}
\end{center}
This problem is on the spot since the very first versions of Theorem \ref{TUB} have been
proved, and in fact it has been the engine for far reaching developments. In 1985,
Lange and Rabinovich \cite{LangeR} studied a particular subalgebra of band-dominated
operators, the so called Wiener algebra, and initiated a series of works
\cite{RochFredTh, LimOps, Marko, LiChW, MarkoWiener} which show that for such
operators the question can be answered affirmatively. Another special class where
the problem is solved is the set of slowly oscillating band-dominated operators
which was treated by Rabinovich, Roch and Silbermann \cite{RochFredTh} using
localizing techniques \cite{RaRo}, which lead to a well-engineered theory on local
operator spectra.
A third class of operators for which the problem was solved in \cite{RochOnFS}
are the block diagonal operators formed by the sequence of so-called finite sections of another
band-dominated operator $B$.

Notice that the story of limit operators actually had begun much earlier
in the late 1920's in Favard's paper \cite{Favard} for studying ODEs with
almost-periodic coefficients, and in the work of Muhamadiev \cite{Muh1, Muh2, Muh3, Muh4}.
A review of this history is, for example, in the introduction of \cite{LiChW}.
A comprehensive presentation of these results, further achievements and applications
e.g. to convolution and pseudo-differential operators, as well as the required tools,
can be found in the 2004 book \cite{LimOps} of Rabinovich, Roch and Silbermann.

Starting with that book, \cite{LimOps}, the technical framework was extended in a way
that now also the extremal cases $l^\infty$ and $l^1$ \cite{LiSi, DissMarko, Marko, MarkoWiener}
as well as vector-valued $l^p$-spaces could be treated (see \cite{SeFre} for the current state of the art).
After the extension to $l^{\infty}$ and $l^{1}$ it turned out that in these two extremal cases the uniform
boundedness condition in Theorem \ref{TUB} was
redundant \cite{DissMarko, Marko}
(note that for $l^\infty$ this implicitly occurs
in \cite{RochFredTh} already). In $l^{\infty}$, results got even better when previous
techniques were combined with the generalized collectively compact operator theory of
Chandler-Wilde and Zhang \cite{ChWZh}. This led to a further simplification of the
Fredholm criteria to just requiring the injectivity of the limit operators
\cite{LiChWFavard, HabilMarko, LiChW}, which is known as Favard's condition.

Another natural playground for the limit operator method, which is not subject of the
present text, is the stability analysis for approximation methods applied to
band-dominated and related operators (cf. \cite{NonStrongly, RaRoSiAlgFSM, RaRoSiL2,
RochFredTh, LiEssBdd, HabilMarko, SeSi2, SeSi3, SeSi4, SeDis}, again the books
\cite{HaRoSi, Standard, LimOps, Marko, RochGross, LiChW} and the references cited there).

\medskip

Throughout all these works the ``big question'' has been ubiquitous e.g. in the
introductions of the books \cite{LimOps, Marko}, the open problems section of
\cite{LiChW}, the open questions in \cite[Section 6]{RochFredTh}, the discussion of
sufficient or weakly sufficient families of homomorphisms in \cite{Standard}, the
collection of main results in \cite[Section 2]{RochGross}, the papers
\cite{RochFredTh, RaRoSiAlgFSM, RaRoSiL2, RochOnFS} and many others.
In his review of the article \cite{RaRo}, A. B\"ottcher writes about the
uniform boundedness condition in Theorem \ref{TUB}: ``Condition (*) is nasty to work with.''
There is nothing to add to this.

\medskip

The aim of this paper is to solve the problem in general and to give an affirmative answer
to the big question (cf. Theorem \ref{CBig}) by showing that the ``nasty condition'' is redundant.
For this, we start with
recalling the required definitions and basic results in the next section.

In Section 3 the main results are proved. In short, the golden thread there is as
follows: We show that the lower norm of a band-dominated operator $A$ can be
approximated by a sort of restricted lower norms which are defined with respect to
elements of uniformly bounded support. Moreover, this approximation turns out to extend
uniformly to all operators in $\sigma_{\op}(A)$. Finally, for rich band-dominated $A$,
this permits to identify an element in the operator spectrum $\sigma_{\op}(A)$ which
realizes the infimum of the lower norms of all limit operators of $A$.
This removes the uniform boundedness condition from Theorem \ref{TUB} and allows to write
the $\Pc$-essential spectrum %of a rich band-dominated operator
as in Corollary \ref{cor:ess.spec}.

The final Section 4 shines a light on the importance of the prerequisites \textit{rich}
and \textit{band-dominated} in the main results.

%%%%%%%%%%%%%%%%%%%%%%%%%%%%%%%%%%%%%%%%%%%%%%%%%%%%%%%%%%%%%%%%%%%%%%%%%%%%%%%
%%%%%%%%%%%%%%%%%%%%%%%%%%%%%%%%%%%%%%%%%%%%%%%%%%%%%%%%%%%%%%%%%%%%%%%%%%%%%%%
%%%%%%%%%%%%%%%%%%%%%%%%%%%%%%%%%%%%%%%%%%%%%%%%%%%%%%%%%%%%%%%%%%%%%%%%%%%%%%%
\section{\texorpdfstring{$\Pc$}{P}-Fredholmness and limit operators}
%%%%%%%%%%%%%%%%%%%%%%%%%%%%%%%%%%%%%%%%%%%%%%%%%%%%%%%%%%%%%%%%%%%%%%%%%%%%%%%
%%%%%%%%%%%%%%%%%%%%%%%%%%%%%%%%%%%%%%%%%%%%%%%%%%%%%%%%%%%%%%%%%%%%%%%%%%%%%%%
%%%%%%%%%%%%%%%%%%%%%%%%%%%%%%%%%%%%%%%%%%%%%%%%%%%%%%%%%%%%%%%%%%%%%%%%%%%%%%%

In all what follows we let $\Xb$ stand for a (generalized) sequence space $l^p(\Zb^N,X)$
with parameters $p\in\{0\}\cup[1,\infty]$, $N\in\Nb$ and a Banach space $X$.
These (generalized) sequences are of the form $x=(x_i)_{i\in\Zb^N}$ with
all $x_i\in X$. The spaces are equipped with the usual $p$-norm. In our notation,
$l^0(\Zb^N,X)$ stands for the closure in $l^\infty(\Zb^N,X)$ of the set of all
sequences $(x_i)_{i\in\Zb^N}$ with finite support.
Notice that this model in particular covers the spaces $L^p(\Rb^N)$ by a natural
isometric identification (see \eqref{eq:Lp-as-lp} below). For a subset
$F\subset\Zb^N$ we denote its characteristic function by $\chi_F$.

%%%%%%%%%%%%%%%%%%%%%%%%%%%%%%%%%%%%%%%%%%%%%%%%%%%%%%%%%%%%%%%%%%%%%%%%%%%%%%%
\subparagraph{Operators and convergence}
%%%%%%%%%%%%%%%%%%%%%%%%%%%%%%%%%%%%%%%%%%%%%%%%%%%%%%%%%%%%%%%%%%%%%%%%%%%%%%%

The following definitions and results are taken from e.g. \cite{LimOps, Marko, SeSi2, SeFre}.
Starting with the sequence $\Pc=(P_n)$ of the canonical projections
\[P_n:\Xb\to\Xb,\quad (x_i)_{i\in\Zb^N}\mapsto(\chi_{\{-n,\ldots,n\}^N}(i)\,x_i)_{i\in\Zb^N}\]
one says that a bounded linear operator $K$ on $\Xb$ is $\Pc$-compact if
\[\|(I-P_n)K\|+\|K(I-P_n)\|\to 0\quad\text{as}\quad n\to\infty.\]
The set of all $\Pc$-compact operators is denoted by $\pc{\Xb}$. Moreover, we
introduce the subset $\pb{\Xb}$ of $\lb{\Xb}$, the set of all bounded linear operators on
$\Xb$, as follows:
\[\pb{\Xb}:=\{A\in\lb{\Xb}:AK,KA\in\pc{\Xb}\text{ for all }K\in\pc{\Xb}\}.\]
One may say that those operators are compatible with $\pc{\Xb}$ and, indeed, $\pb{\Xb}$
is a closed subalgebra of $\lb{\Xb}$ and $\pc{\Xb}$ forms a closed two sided ideal in
$\pb{\Xb}$.

An operator $A\in\pb{\Xb}$ is said to be $\Pc$-Fredholm if there is an operator
$B\in\lb{\Xb}$ such that $AB-I$ and $BA-I$ are $\Pc$-compact. By \cite[Theorem 1.16]{SeSi3}
$B$ belongs to $\pb{\Xb}$ in this case, i.e. $\Pc$-Fredholmness of $A$ is equivalent to
the invertibility of $A+\pc{\Xb}$ in $\pb{\Xb}/\pc{\Xb}$.

Say that a bounded sequence $(A_n)\subset\pb{\Xb}$ converges $\Pc$-strongly to an
operator $A\in\lb{\Xb}$ if
\[\|P_m(A_n-A)\|+\|(A_n-A)P_m\|\to0\quad\text{as}\quad n\to\infty\]
for every $m\in\Nb$.
We shortly write $A_n\pto A$ in that case and note that $A\in\pb{\Xb}$
\begin{rem}
Notice that for all $p\in(1,\infty)$ and $\dim X<\infty$ these $\Pc$-notions
coincide with the classical ones: $\pc{\Xb}=\lc{\Xb}$ the set of compact operators,
$\pb{\Xb}=\lb{\Xb}$, an operator is $\Pc$-Fredholm if and only if it is Fredholm, and
a sequence $(A_n)$ converges $\Pc$-strongly to $A$ if and only if $A_n\to A$ and
$A_n^*\to A^*$ strongly.

The reason for the definition of the $\Pc$-notions is simply to extend the
concepts, tools and connections between the classical notions, which are well known
from standard Functional Analysis.
\end{rem}

%%%%%%%%%%%%%%%%%%%%%%%%%%%%%%%%%%%%%%%%%%%%%%%%%%%%%%%%%%%%%%%%%%%%%%%%%%%%%%%
\subparagraph{Band-dominated operators}
%%%%%%%%%%%%%%%%%%%%%%%%%%%%%%%%%%%%%%%%%%%%%%%%%%%%%%%%%%%%%%%%%%%%%%%%%%%%%%%
Every sequence $a=(a_n)\in l^\infty(\Zb^N,\lb{X})$ gives rise to an operator
$aI\in\lb{\Xb}$, a so-called multiplication operator, via the rule
$(ax)_i=a_ix_i$, $i\in\Zb^N.$
%Evidently, $\|aI\|_{\lb{\Xb}}=\|a\|_{l^\infty}$.
For every $\alpha\in\Zb^N$, we define the shift operator
$V_\alpha:\Xb\to\Xb$, $(x_i)\mapsto(x_{i-\alpha})$.

A band operator is a finite sum of the form $\sum a_\alpha V_\alpha$,
where $a_\alpha I$ are multiplication operators.
In terms of the generalized matrix-vector multiplication
\[
(a_{ij})_{i,j\in\Zb^N}\ (x_j)_{j\in\Zb^N}=(y_i)_{i\in\Zb^N}
\quad\textrm{with}\quad
y_i\ =\ \sum_{j\in\Zb^N}a_{ij}x_j,\ i\in\Zb^N,\quad\textrm{where\ } a_{ij}\in\lb{X},
\]
band operators $A$ act on $\Xb=l^p(\Zb^N,X)$ via multiplication by band matrices $(a_{ij})$,
that means $a_{ij}=0$ if $|i-j|$ exceeds the so-called band-width of $A$.

Typical examples are discretizations of differential operators on $\Rb^N$, such as discrete Schr\"o\-dinger operators.
(There is large interest in spectral properties of discrete Schr\"odinger and more general Jacobi operators.
Here are some references that involve limit-operator-type arguments:
\cite{CarmonaLacroix91,PasturFigotin,Davies,LastSimon06,Remling2,RaRo:disc.struc}.)

In many physical models however, interaction $a_{ij}$ between data at locations $i$ and $j$ decreases in a certain way
as $|i-j|\to\infty$ rather than suddenly stop at a prescribed distance of $i$ and $j$. This is one of the reasons for introducing the following notion:

An operator is called band-dominated if it is the limit, w.r.t. the
operator norm, of a sequence of band operators. Here typical examples are integral operators
\begin{equation} \label{eq:iop}
A:f\mapsto\int_{\Rb^{N}}k(\cdot,t)f(t)\,dt
\end{equation}
on $L^{p}(\Rb^{N})$, where the kernel function $k(s,t)$ decays sufficiently fast as $|s-t|\to\infty$.
Such operators are often connected with boundary integral problems (see e.g. \cite{Jorgens,AnseloneSloan,NumAna,Agarwal,ChWZh,ArensCWHas2,LimOps,LpInd,LiChWRough}).
Via the isometric isomorphism $L^{p}(\Rb^{N})\to l^{p}(\Zb^{N},L^{p}([0,1]^{N}))$ that sends $f$ to
\begin{equation}\label{eq:Lp-as-lp}
\textrm{the sequence}\ (f_i)_{i\in\Zb^N}\ \textrm{of restrictions}\ f_i:=f|_{i+[0,1]^N}\in L^{p}(i+[0,1]^{N})\cong L^{p}([0,1]^{N})=:X,
\end{equation}
the integral operator \eqref{eq:iop} corresponds to an infinite matrix $(a_{ij})_{i,j\in\Zb^N}$ with integral operator entries $a_{ij}:L^{p}(j+[0,1]^{N})\cong X\to L^{p}(i+[0,1]^{N})\cong X$.

We denote the class of all band-dominated operators on $\Xb=l^p(\Zb^N,X)$ by $\Alp$.
In contrast to the set of all band operators
(which is an algebra but not closed in $\lb{\Xb}$), the set $\Alp$ is a Banach algebra,
for which the inclusions $\pc{\Xb}\subset\Alp\subset\pb{\Xb}\subset\lb{\Xb}$ hold.

The first studies of particular subclasses of band-dominated operators and their Fredholm properties were for
the case of constant matrix diagonals, that is when the matrix entries $a_{ij}$ only depend on the difference $i-j$,
so that $A$ is a convolution
operator (a.k.a. Laurent or bi-infinite Toeplitz matrix, the stationary case) \cite{GoKre,Sim,GoFe,HaRoSi,Analysis}. Subsequently, the focus went to more general classes, such as convergent, periodic and almost periodic matrix diagonals, until at the current point arbitrary matrix diagonals can be studied -- as long as they are bounded. This possibility is due to the notion of limit operators that enables evaluation of the asymptotic behavior of an operator $A$ even for merely bounded diagonals in the matrix $(a_{ij})$.

%%%%%%%%%%%%%%%%%%%%%%%%%%%%%%%%%%%%%%%%%%%%%%%%%%%%%%%%%%%%%%%%%%%%%%%%%%%%%%%
\subparagraph{Limit operators}
%%%%%%%%%%%%%%%%%%%%%%%%%%%%%%%%%%%%%%%%%%%%%%%%%%%%%%%%%%%%%%%%%%%%%%%%%%%%%%%
Say that a sequence $h=(h_n)\subset \Zb^N$ tends to infinity if $|h_n|\to\infty$ as
$n\to\infty$.
If $h=(h_n)\subset \Zb^N$ tends to infinity and $A\in\pb{\Xb}$ then
\[A_h:=\plimn V_{-h_n}AV_{h_n},\]
if it exists, is called the limit operator of $A$ w.r.t. the sequence $h$.
The set
\[\sigma_{\op}(A):=\{A_h:h=(h_n)\subset\Zb^N\text{ tends to infinity}\}\]
of all limit operators of $A$ is called its operator spectrum.

We say that $A\in\pb{\Xb}$ has a rich operator spectrum (or we simply call $A$ a rich operator)
if every sequence $h\subset\Zb^N$ tending to infinity has a subsequence $g\subset h$ such that
the limit operator  $A_g$ of $A$ w.r.t. $g$ exists. The set of all rich operators is denoted
by $\pbr{\Xb}$ and the set of all rich band-dominated operators by $\Alpr$. Recall from
\cite[Corollary 2.1.17]{LimOps} or \cite[Corollary 3.24]{Marko} that $\Alpr=\Alp$ whenever $\dim X<\infty$.

\begin{prop}\label{PRichOps}
(\cite[Proposition 3.7]{DissMarko} or \cite[Corollary 1.2.7]{LimOps}; and \cite[Corollary 17]{SeFre})\\
Let again $\Xb=l^p(\Zb^N,X)$, where $p\in\{0\}\cup[1,\infty]$, $N\in\Nb$ and $X$ is a Banach space.
\begin{itemize}
\item For every $A\in\pbr{\Xb}$, $\sigma_{\op}(A)$ is sequentially compact
			with respect to $\Pc$-strong convergence.
\item The set $\pbr{\Xb}$ forms a closed subalgebra of $\pb{\Xb}$ and contains
			$\pc{\Xb}$ as a closed two-sided ideal.
\item If $p<\infty$ and $A\in\pbr{\Xb}$, then $A^*\in\Lc^\$(\Xb^*,\Pc)$ and
			\begin{equation}
			\sigma_{\op}(A^*)=(\sigma_{\op}(A))^*:=\{B^*:B\in\sigma_{\op}(A)\}.
			\end{equation}
\end{itemize}
\end{prop}
\noindent
Note that Proposition \ref{PRichOps} translates to $\Alpr$ as well and that
$\sigma_{\op}(A)\subset\Alp$ if $A\in\Alp$.
Finally recall from the introduction that, for operators in $\Alpr$, Theorem \ref{TUB} is in force
and provides a characterization of their $\Pc$-Fredholm property.

%%%%%%%%%%%%%%%%%%%%%%%%%%%%%%%%%%%%%%%%%%%%%%%%%%%%%%%%%%%%%%%%%%%%%%%%%%%%%%%
%%%%%%%%%%%%%%%%%%%%%%%%%%%%%%%%%%%%%%%%%%%%%%%%%%%%%%%%%%%%%%%%%%%%%%%%%%%%%%%
%%%%%%%%%%%%%%%%%%%%%%%%%%%%%%%%%%%%%%%%%%%%%%%%%%%%%%%%%%%%%%%%%%%%%%%%%%%%%%%
\section{The main results}
%%%%%%%%%%%%%%%%%%%%%%%%%%%%%%%%%%%%%%%%%%%%%%%%%%%%%%%%%%%%%%%%%%%%%%%%%%%%%%%
%%%%%%%%%%%%%%%%%%%%%%%%%%%%%%%%%%%%%%%%%%%%%%%%%%%%%%%%%%%%%%%%%%%%%%%%%%%%%%%
%%%%%%%%%%%%%%%%%%%%%%%%%%%%%%%%%%%%%%%%%%%%%%%%%%%%%%%%%%%%%%%%%%%%%%%%%%%%%%%
\begin{defn} \label{def:nu}
Given two Banach spaces $Y, Z$ the lower norm of an operator $A\in\lb{Y,Z}$ is
\[\nu(A):=\inf\{\|Ay\|:y\in Y, \|y\|=1\}.\]
For operators $A\in\lb{\Xb}$ and subsets $F\subset\Zb^N$ we
shortly write $\nu(A|_F)$ for the lower norm of the restricted operator
$A|_F:=A\chi_FI:\im \chi_FI\to\Xb$.
\end{defn}

If $A$ is invertible then $\nu(A)=1/\|A^{-1}\|$. So our big question translates:
If every $B\in\sigma_{\op}(A)$ is invertible, are then the lower norms $\nu(B)$
uniformly bounded away from zero? In short:
\begin{equation*} \tag{BQ} \label{eq:BQ1}
(\forall B\in \sigma_{\op}(A):\ B\textrm{ is invertible})\quad\stackrel{?}
\Longrightarrow\quad\inf\{\nu(B):B\in\sigma_{\op}(A)\}>0
\end{equation*}
We prove that the answer is `yes'. Our approach uses the sequential compactness
property of $\sigma_{\op}(A)$ as mentioned in Proposition \ref{PRichOps}. This
property motivates the following strategy of proof: Suppose all $B\in\sigma_{\op}(A)$
were invertible but there was a sequence $B_1, B_2, ... \in\sigma_{\op}(A)$ with
$\nu(B_n)\to 0$. By our compactness, there is a subsequence $(B_{n_k})$ with limit
$B\in\sigma_{\op}(A)$. Can $B$ be invertible?

In case $\|B_{n_k}-B\|\to 0$ one can conclude $\nu(B_{n_k})\to\nu(B)$ and hence $\nu(B)=0$,
which contradicts the invertibility of $B$ (proof finished).
Unfortunately, all we have, by Proposition \ref{PRichOps}, is $B_{n_k}\pto B$,
which does not imply $\nu(B_{n_k})\to\nu(B)$. As an example, look at
$B_n:=P_n+\frac 1n(I-P_n)\pto I=:B$, where the ``region'',
$\Zb^N\setminus\{-n,...,n\}^N$, that is responsible for the small values of $\nu(B_n)=\frac 1n$,
moves away from the origin as $n\to\infty$ whence it has no impact on $\plim B_n$ and its
lower norm.

To get back on track with our proof, we will
\begin{itemize}
\item pass to translates $V_{-j}B_{n_k}V_j$ of $B_{n_k}$ such that the ``bad region'' of
$B_{n_k}$ shifts to the origin and hence contributes to the lower norm of the $\Pc$-limit $B$.
(Note that obviously $V_{-j}BV_j\in\sigma_{\op}(A)$ if $B\in\sigma_{\op}(A)$ and $j\in\Zb^N$.)
\item study the lower norms $\nu(B|_F)$ from Definition \ref{def:nu}, in particular
for bounded sets $F\subset\Zb^{N}$. (Note that $B_{n_k}\pto B$ implies $\|(B_{n_k}-B)\chi_FI\|\to 0$ and hence $\nu(B_{n_k}|_F)\to\nu(B|_F)\ge\nu(B)$ for all bounded sets $F\subset\Zb^N$.)
\item in order to find appropriate shifts $V_j$ and bounded sets $F$, study the lower norms of operators,
restricted to elements $x\in\Xb$ with support in any set of a prescribed diameter $D$ (i.e. in an arbitrary translate of the cube $[-\frac D2,\frac D2]^{N}$) -- see Definition \ref{def:nuD} below.
\end{itemize}

\begin{defn}\label{def:nuD}
The support of a sequence $x=(x_n)\in \Xb$ is the set
$\supp x :=\{n\in\Zb^N:x_n\neq 0\}$. The diameter of a subset $U\subset \Zb^N$
is defined as $\diam U := \sup\{|u-v|_\infty:u,v\in U\}$,
where $|\cdot|_\infty$ denotes the maximum norm on $\Rb^N$.
Moreover, put $\dist(U,V):=\min\{|u-v|_\infty:u\in U, v\in V\}$ for two nonempty
sets $U,V\subset\Zb^N$.
For $A\in\lb{\Xb}$, $F\subset\Zb^N$ and $D\in\Nb$, we now put
\[\nu_D(A|_F):=\inf\{\|A|_Fx\|:x\in \im\chi_FI, \|x\|=1, \diam\supp x \leq D\}.\]
\end{defn}

We first show that $\nu_D(A|_F)\to\nu(A|_F)$ as $D\to\infty$, uniformly in a surprisingly general sense.
We say $A$ has band-width less than $w\in\Nb$
if $\chi_U A\chi_V I=0$ for all $U,V\subset\Zb^N$ with $\dist(U,V)>w$.

\begin{prop} \label{prop:5}
Let $\delta>0$, $r>0$ and $w\in\Nb$. Then there is a $D\in\Nb$
such that for all band operators $A$ with $\|A\|<r$ and band-width less than $w$
\[\nu(A|_F)+\delta\geq\nu_D(A|_F)\geq\nu(A|_F)
	\quad\text{for all}\quad F\subset\Zb^N.\]
\end{prop}
\begin{proof}
Clearly $\nu_D(\cdot)\geq\nu(\cdot)$ holds.
Let $R:=4w$ and for every function $\varphi$ on $\Rb^N$ and every
$\alpha\in\Zb^N$ we set $\varphi^{\alpha,R}(s):=\varphi\left(s/R-\alpha\right)$.
We start with the case $N=1$ and $p\in[1,\infty)$. We define
$\chi:=\chi_{[-\frac{1}{2},\frac{1}{2})}$, and for $x\in \im\chi_FI$
and $m\in\Nb$ we consider the decomposition
\begin{align*}
x = \sum_{l=0}^{m-1}\sum_{n\in m\Zb}\chi^{l+n,R}x
 %&= \sum_{l=0,\,l\neq l_0}^{m-1}\sum_{n\in m\Zb}\chi^{l+n,R}x
			%+\sum_{n\in m\Zb}\chi^{l_0+n,R}x\\
 &= \sum_{j=1}^{m-1}\sum_{n\in m\Zb}\chi^{j+l_0+n,R}x
			+\sum_{n\in m\Zb}\chi^{l_0+n,R}x
\end{align*}
where  $l_0\in\{0,\ldots,m-1\}$ is chosen in such a way, that the norm of
$y:=\sum_{n\in m\Zb}\chi^{l_0+n,R}x$ is minimal. Then, in particular
$\|y\|\leq m^{-1/p}\|x\|$ and
\[\|Ax\|
%		\geq  \left\|\sum_{j=1}^{m-1}\sum_{n\in m\Zb}A\chi^{j+l_0+n,R}x\right\|
%			-\left\|\sum_{n\in m\Zb}A\chi^{l_0+n,R}x\right\|
			\geq  \left\|\sum_{j=1}^{m-1}\sum_{n\in m\Zb}A\chi^{j+l_0+n,R}x\right\|
			-\|A\|\|y\|
			\geq \left\|\sum_{n\in m\Zb}\varphi_n\sum_{j=1}^{m-1}A\chi^{j+l_0+n,R}x\right\|
            -r\|y\|\]
%			-rm^{-1/p}\|x\|\]
where $\varphi_n:=\sum_{k=0}^{m-1}\chi^{k+l_0+1/2+n,R}$ can be inserted since the
band-width of $A$ is less than $w=R/4$.
Further
\begin{align*}
&\left\|\sum_{n\in m\Zb}\varphi_n\sum_{j=1}^{m-1}A\chi^{j+l_0+n,R}x\right\|^p
	= \sum_{n\in m\Zb}\left\|\varphi_n\sum_{j=1}^{m-1}A\chi^{j+l_0+n,R}x\right\|^p
	= \sum_{n\in m\Zb}\left\|\sum_{j=1}^{m-1}A\chi^{j+l_0+n,R}x\right\|^p\\
	&\quad\quad\quad\quad\quad\quad
	\geq \sum_{n\in m\Zb}\nu_{mR}(A|_F)^p\left\|\sum_{j=1}^{m-1}\chi^{j+l_0+n,R}x\right\|^p
	= \nu_{mR}(A|_F)^p\left\|\sum_{n\in m\Zb}\sum_{j=1}^{m-1}\chi^{j+l_0+n,R}x\right\|^p\\
	&\quad\quad\quad\quad\quad\quad\geq \nu_{mR}(A|_F)^p(\|x\| - \|y\|)^p.
\end{align*}
Combining these estimates we arrive at
\begin{align*}
\|Ax\|
%&\geq  \nu_{mR}(A|_F)\left\|\sum_{n\in m\Zb}\sum_{j=1}^{m-1}\chi^{j+l_0+n,R}x\right\|-\|A\|\|y\|\\
&\geq \nu_{mR}(A|_F)(\|x\| - \|y\|)-r\|y\|
\geq \nu_{mR}(A|_F)\|x\| - 2rm^{-1/p}\|x\|.
\end{align*}
Taking the infimum over all $x\in\im\chi_FI$, $\|x\|=1$ this yields
$\nu(A|_F)\geq\nu_{mR}(A|_F)-2rm^{-1/p},$
hence the assertion with $m>(2r\delta^{-1})^p$ and $D:=mR=4mw$.

If $N>1$, then we define functions
\[\chi_1:=\chi_{[-\frac{1}{2},\frac{1}{2})\times\Rb^{N-1}},\quad
  \chi_2:=\chi_{\Rb\times [-\frac{1}{2},\frac{1}{2})\times\Rb^{N-2}},\quad \ldots,\quad
	\chi_N:=\chi_{\Rb^{N-1}\times [-\frac{1}{2},\frac{1}{2})}.\]
For $x\in\chi_FI$, applying the above estimates $N$ times we get
\begin{align*}
\|Ax\|&\geq  \left\|\sum_{j_1=1}^{m-1}\sum_{n_1\in m\Zb}A\chi_1^{j_1+l_1^0+n_1,R}x\right\|
			-rm^{-1/p}\|x\|\\
	&\geq\ldots\\		
	&\geq  \left\|\sum_{(j_1,\ldots,j_N)\in\{1,\ldots,m-1\}^N}\sum_{(n_1,\ldots,n_N)\in (m\Zb)^N}
		A\chi_1^{j_1+l_1^0+n_1,R}\cdots\chi_N^{j_N+l_N^0+n_N,R}x\right\|
			-Nrm^{-1/p}\|x\|
\end{align*}
and with $\varphi_{(n_1,\ldots,n_N)}:=\sum_{(k_1,\ldots,k_N)\in\{0,\ldots,m-1\}^N}
	\chi_1^{k_1+l_1^0+1/2+n_1,R}\cdots\chi_N^{k_N+l_N^0+1/2+n_N,R}$
the first summand again turns out to be %not smaller than
\begin{align*}
%&\left\|\sum_{(j_1,\ldots,j_N)\in\{1,\ldots,m-1\}^N}\sum_{(n_1,\ldots,n_N)\in (m\Zb)^N}
		%A\chi_1^{j_1+l_1^0+n_1,R}\cdots\chi_N^{j_N+l_N^0+n_N,R}x\right\|^p\\
&\quad \geq		
\nu_{mR}(A|_F)^p\left\|\sum_{(j_1,\ldots,j_N)\in\{1,\ldots,m-1\}^N}\sum_{(n_1,\ldots,n_N)\in (m\Zb)^N}
		\chi_1^{j_1+l_1^0+n_1,R}\cdots\chi_N^{j_N+l_N^0+n_N,R}x\right\|^p,
\end{align*}
hence the assertion again follows by
\[\|Ax\|\geq \nu_{mR}(A|_F)\|x\| - 2Nrm^{-1/p}\|x\|.\]

Finally the case $p\in\{0,\infty\}$: This time put $\chi:=\chi_{[-1,1]^N}$ and $\psi(t):=\max(1-|t|_\infty,0)$ for $t\in\Rb^N$. For $\alpha,\beta\in\Zb^N$ and $k\in\Nb$ we note that
\[
\chi^{\beta/w,w}\psi^{\alpha,k}=\psi^{\alpha,k}(\beta)\chi^{\beta/w,w}+\varphi\quad\textrm{with}\quad
\|\varphi\|_\infty=\sup\{|\varphi(t)|:|t-\beta|_\infty\le w\}\le \frac wk
\]
and $\chi_{\{\beta\}}A=\chi_{\{\beta\}}A\chi^{\beta/w,w}$ since $A$ has band-width less than $w$. We conclude, for all $x\in\Xb$,
\begin{align*}
\|\chi_{\{\beta\}}A\psi^{\alpha,k}x\|
&= \|\chi_{\{\beta\}}A\chi^{\beta/w,w}\psi^{\alpha,k}x\|
 \le \psi^{\alpha,k}(\beta)\|\chi_{\{\beta\}}A\chi^{\beta/w,w}x\|+\|\chi_{\{\beta\}}A\varphi x\|\\
&\le 1\cdot\|\chi_{\{\beta\}}Ax\|+\frac wk \|A\|\|x\| \le \|Ax\|+\frac {wr}k \|x\|
\end{align*}
since $\|A\|<r$. Taking the supremum over all $\beta\in\Zb^N$ yields $\|A\psi^{\alpha,k}x\|\le \|Ax\|+\frac{wr}k \|x\|$, and hence,
for all $x\in\im\chi_F I$ with $\|x\|=1$: $\nu_{2k}(A|_F)\|\psi^{\alpha,k}x\|\le \|Ax\|+\frac{wr}k$. Now taking the supremum over all $\alpha\in\Zb^N$ yields $\nu_{2k}(A|_F)\le \|Ax\|+\frac{wr}k$. Finally take the infimum over all $x\in\im\chi_F I$ with $\|x\|=1$ to get $\nu_{2k}(A|_F)\le \nu(A|_F)+\frac{wr}k$ and make $k$ sufficiently large for $\frac{wr}k<\delta$.
\end{proof}

This result is going to be a main ingredient of our plot. The basic idea in the proof
presented above for Proposition \ref{prop:5}, namely to split the domain of a band
operator $A$ into separated regions and to exploit that the action
of $A$ is limited by its band-width, has a long and successful tradition
and has lead to deep results on inverse closedness and on the existence of
$\Pc$-regularizers in $\pb{\Xb}$ or $\Alp$ (see \cite{KozakSimo, Analysis, NumAna,
LimOps, SeDis, SeSi3}). We also point out that the proof reveals explicit formulas
for the choice of $D$ depending on $\delta$, $r$ and $w$.
In addition, we would like to show an alternative proof
for the case $p\in [1,\infty)$ that is based on an idea of \cite{CW.Heng.ML:UpperBounds},
where it is used to bound and approximate the pseudo-spectrum of $A$ by a union of
pseudo-spectra of (moderately sized) finite sections of $A$.
\begin{proof}
Let $A$ be a band operator of the norm $\|A\|<r$ and its band-width less than $w\in\Nb$.
%, i.e.\! $\chi_U A\chi_V I=0$ for all $U,V\subset\Zb^N$ with $\dist(U,V)>w$.

For $n\in\Nb$ and $k\in\Zb^N$, put $C_n:=\{-n,...,n\}^N$, $C_{n,k}:=k+C_n$, $D_n:=C_{n+w}\setminus C_{n-w}$, $D_{n,k}:=k+D_n$,
$c_n:=|C_n|=|C_{n,k}|=(2n+1)^N$ and $d_n:=|D_n|=|D_{n,k}|=c_{n+w}-c_{n-w}\sim n^{N-1}$.
Abbreviate $\chi_{C_{n,k}}I=:P_{n,k}$ and $\chi_{D_{n,k}}I=:\Delta_{n,k}$, and note the following facts:
\begin{itemize}
\item[(a)] For all finite sets $S\subset\Zb^N$ and all $x\in\Xb$, it holds
$
\displaystyle\sum_{k\in\Zb^N}\|\chi_{k+S}\,x\|^p = |S|\cdot \|x\|^p.
$
\item[(b)] For the commutator $[P_{n,k},A]:=P_{n,k}A-AP_{n,k}$, one has $[P_{n,k},A]=[P_{n,k},A]\Delta_{n,k}$, so that for all $x\in\Xb$,
$
\|[P_{n,k},A]x\| = \|[P_{n,k},A]\Delta_{n,k}x\| \le \|[P_{n,k},A]\| \|\Delta_{n,k}x\|  < 2r \|\Delta_{n,k}x\|
$
and hence
\[
\sum_{k\in\Zb^N} \|[P_{n,k},A]x\|^p \le \sum_{k\in\Zb^N} 2^pr^p \|\Delta_{n,k}x\|^p \stackrel{\textrm{(a)}}{=} 2^p r^p d_n \|x\|^p.
\]
\item[(c)] For all $a,b,\varphi>0$, one has
$
(a+b)^p \le (1+\varphi)^{p-1}a^p+(1+\varphi^{-1})^{p-1}b^p,
$
with equality iff $\varphi=\frac ba$. (This is a simple calculus exercise: minimize the right-hand side as a function of $\varphi$.)
\end{itemize}
Now, to our arbitrary $\delta>0$, choose $n_0\in\Nb$ large enough that $\frac{d_n}{c_n}<(\frac \delta {4r})^p$ for all $n\ge n_0$. Moreover, given an arbitrary $F\subset\Zb^{N}$, fix $x\in\im\chi_{F}I$ such that $\|Ax\|<(\nu(A|_F)+\frac \delta 2)\|x\|$.
Then we conclude as follows:
\begin{align*}
\sum_{k\in\Zb^N} \|AP_{n,k}x\|^p
&\stackrel{\textrm{\makebox[15pt]{}}}{\le} \sum_{k\in\Zb^N} \left(\|P_{n,k}Ax\| + \|[P_{n,k},A]x\|\right)^p\\
&\stackrel{\textrm{\makebox[15pt]{(c)}}}{\le} (1+\varphi)^{p-1} \sum_{k\in\Zb^N} \|P_{n,k}Ax\|^p
 + (1+\varphi^{-1})^{p-1} \sum_{k\in\Zb^N} \|[P_{n,k},A]x\|^p,\quad \forall\varphi>0\\
&\stackrel{\textrm{\makebox[15pt]{(a),(b)}}}{\le} (1+\varphi)^{p-1} c_n\|Ax\|^p
 + (1+\varphi^{-1})^{p-1} 2^pr^p d_n \|x\|^p,\qquad \forall\varphi>0.
\end{align*}
Minimizing the latter term over all $\varphi>0$ we get by the fact $\textrm{(c)}$ that
\begin{align*}
\sum_{k\in\Zb^N} \|AP_{n,k}x\|^p
&\le \left(c_n^{1/p}\,\|Ax\| + d_n^{1/p}\,2\,r\,\|x\|\right)^p \\ %\textrm{for a particular $\varphi=\frac ba$}\\
&\le \Big(\nu(A|_F) + \frac\delta 2 + \underbrace{\left(\frac{d_n}{c_n}\right)^{1/p}2r}_{<\frac\delta2\ {\rm if}\ n\ge n_0}\Big)^p\underbrace{c_n\|x\|^p}_{{\rm use\ (a)}}
< (\nu(A|_{F}) + \delta)^p \sum_{k\in\Zb^N} \|P_{n,k}x\|^p
\end{align*}
for all $n\ge n_0$. The last inequality shows that there must be some $k\in\Zb^N$ for which
\[
\|AP_{n,k}x\|^p < (\nu(A|_{F}) + \delta)^p\|P_{n,k}x\|^p,
\quad\textrm{and hence}\quad \nu_{2n}(A|_{F}) \le \frac{\|AP_{n,k}x\|}{\|P_{n,k}x\|} < \nu(A|_{F})+\delta
\]
for all $n\ge n_0$ since $x\in\im\chi_{F}I$. This finishes the proof.
\end{proof}

Next we observe that the previous result extends to band-dominated operators
and, actually, even to their operator spectra in a uniform manner.
\begin{cor}\label{CLower}
Let $A\in\Alp$ and $\delta>0$. Then there is a $D\in\Nb$ such that
\[\nu(B|_F)+\delta\geq\nu_{\tilde{D}}(B|_F)\geq\nu(B|_F)
	\quad\text{for all}\quad \tilde{D}\geq D,\quad F\subset\Zb^N \quad\text{and}
	\quad B\in\{A\}\cup\sigma_{\op}(A).\]
\end{cor}
\begin{proof}
Note that the algebra of band operators is dense in the Banach algebra of
band-dominated operators. Thus, given $A$ and $\delta$, there is a band operator
$A'$ such that $\|A-A'\|\leq\delta/3$, hence
$|\nu(A|_F)-\nu(A'|_F)|, |\nu_D(A|_F)-\nu_D(A'|_F)|\leq\|A-A'\|\leq\delta/3$ as well.
By the previous proposition there exists a $D>0$ such that
$\nu(A'|_F)+\delta/3\geq\nu_D(A'|_F)\geq\nu(A'|_F)$ for all $F\subset\Zb^N$.
Since $\nu_D(\cdot)\geq\nu(\cdot)$ always holds, we arrive at
$\nu(A|_F)+\delta\geq\nu_D(A|_F)\geq\nu(A|_F)$ for all $F\subset\Zb^N$.

Now, let $A_g\in\sigma_{\op}(A)$, $F\subset\Zb^N$ and $\epsilon>0$.
By the previous observation, applied to $A_g\in\Alp$, there is a $\hat{D}$ such that
$\nu(A_g|_F)+\epsilon/2 \geq\nu_{\hat{D}}(A_g|_F)$, and moreover there exists an
$x\in\im \chi_F I$, $\|x\|=1$, $\diam\supp x \leq \hat{D}$ with
$\nu_{\hat{D}}(A_g|_F)\geq \|A_g x\|-\epsilon/2$. Choosing $k$ sufficiently large such that
$K:=\{-k,\ldots,k\}^N\supset \supp x$ we find $\nu(A_g|_F)+\epsilon\ge\|A_g x\|\ge\nu(A_g|_{F\cap K})$.
Then
\begin{align*}
\nu(A_g|_F)+\epsilon\geq\nu(A_g|_{F\cap K})
&\geq\nu(V_{-g_n}AV_{g_n}P_k|_{F\cap K})-\|(A_g-V_{-g_n}AV_{g_n})P_k\|\\
&\geq\nu_D(V_{-g_n}AV_{g_n}P_k|_{F\cap K})-\delta-\|(A_g-V_{-g_n}AV_{g_n})P_k\|,
\end{align*}
where the latter tends to $\nu_D(A_gP_k|_{F\cap K})-\delta\geq\nu_D(A_g|_F)-\delta$ as
$n\to\infty$. Since $\epsilon$ is arbitrary, and obviously
$\nu_D(B|_F)\geq\nu_{\tilde{D}}(B|_F)\geq\nu(B|_F)$ holds whenever $\tilde{D}\geq D$,
the assertion follows.
\end{proof}

Now we are ready to prove \eqref{eq:BQ1}.
\begin{thm}\label{TMain}
Let $A\in\Alpr$. Then there exists a
$C\in\sigma_{\op}(A)$ with $\nu(C)=\min\{\nu(B):B\in\sigma_{\op}(A)\}$.
\end{thm}
\begin{proof}
We consider the numbers $\delta_k:=2^{-k}$ and define
$r_l:=\sum_{k=l}^\infty \delta_k=2^{-l+1}$. Then $(r_l)$ forms a
strictly decreasing sequence of positive numbers which tends to $0$.
From Corollary \ref{CLower} we obtain a sequence $(D_k)\subset\Nb$ of
even numbers such that, for every $k\in\Nb$,
\[D_{k+1} > 2 D_k\quad\text{and}\quad
\nu(B|_F) + \delta_k/2 > \nu_{D_k}(B|_F) \quad
\text{for every } B\in\{A\}\cup\sigma_{\op}(A)\text{ and every } F\subset\Nb.\]
Choose a sequence $(B_n)\subset\sigma_{\op}(A)$ such that
$\nu(B_n)\to\inf\{\nu(B):B\in\sigma_{\op}(A)\}$ as $n\to\infty$.
For every $n\in\Nb$ we are going to construct a suitably shifted copy
$C_n\in\sigma_{\op}(A)$ of $B_n$ such that
\begin{equation}\label{EShiftedC}
\nu(C_n|_{F_{3D_l}}) < \nu(B_n) + r_l
\quad\text{for all}\quad l\leq n,
\end{equation}
where we put $F_s:=\{-s/2,\ldots,s/2\}^N$ for every even integer $s$.
For this we firstly note that there exists an $x_n^0\in\Xb$, $\|x_n^0\|=1$,
$\diam\supp x_n^0\leq D_n$ such that
$\|B_nx_n^0\|<\nu_{D_n}(B_n) +\delta_n/2<\nu(B_n)+\delta_n$.
Choose a shift $j_n^0\in\Zb^N$ in order to
centralize $y_n^0:=V_{j_n^0}x_n^0$, which means that $\supp y_n^0\subset F_{D_n}$,
and set $C_n^0:= V_{j_n^0} B_n V_{-j_n^0}\in\sigma_{\op}(A)$. Then
\[\nu(B_n)=\nu(C_n^0)\leq\nu(C_n^0|_{F_{D_n}})\leq\|C_n^0y_n^0\|
=\|B_nx_n^0\|<\nu(B_n)+\delta_n.\]
Next, for $k=1,\ldots,n$, we gradually find $x_n^k$,
 $\|x_n^k\|=1$, $\supp x_n^k\subset F_{D_{n-(k-1)}}$, $\diam\supp x_n^k\leq D_{n-k}$
such that
\begin{align*}
\|C_n^{k-1}x_n^k\|
&<\nu_{D_{n-k}}(C_n^{k-1}|_{F_{D_{n-(k-1)}}})+\delta_{n-k}/2
<\nu(C_n^{k-1}|_{F_{D_{n-(k-1)}}})+\delta_{n-k},
\end{align*}
pass to a centralized copy $y_n^k:=V_{j_n^k}x_n^k$ of $x_n^k$ via a suitable shift
$j_n^k\in F_{D_{n-(k-1)}}$
and then define
$C_n^k:= V_{j_n^k} C_n^{k-1} V_{-j_n^k}\in\sigma_{\op}(A)$. For this operator
we observe
\begin{align*}
\nu(C_n^k|_{F_{D_{n-k}}}) \leq \|C_n^ky_n^k\|=\|C_n^{k-1}x_n^k\|
<\nu(C_n^{k-1}|_{F_{D_{n-(k-1)}}})+\delta_{n-k}.
\end{align*}

\begin{center}
\begin{tikzpicture}[scale=1.0]
   \foreach \x in {-1.5,2.25}\draw[fill, color=gray!50] (\x,-0.1)  -- ++(2.5,0) -- ++(0,0.2)  -- ++(-2.5,0) -- cycle;
   \draw[->] (-7,0)--(7,0);
   \foreach \x in {-6,0,3.75,6} \draw[thick] (\x,-0.2)--(\x,0.2);
   \draw [decorate, decoration={brace, amplitude=12pt}, thick] (6,-1.4)--(-6,-1.4)
       node [midway, below=8pt]{$F_{D_{n-(k-1)}}$};
   \foreach \x in {-2,1.75,2,5.75} \draw (\x,-0.15)--(\x,0.15);
   \draw [decorate, decoration={brace, amplitude=6pt}, thick] (1.0,-0.2)--(-1.5,-0.2)
       node [midway, below=2pt]{$\supp y_{n}^{k}$};
   \draw [decorate, decoration={brace, amplitude=8pt}, thick] (2.0,-0.8)--(-2.0,-0.8)
       node [midway, below=3pt]{$F_{D_{n-k}}$};
   \draw [decorate, decoration={brace, amplitude=6pt}, thick] (2.25,0.2)--(4.75,0.2)
       node [midway, above=4pt]{$\supp x_{n}^{k}$};
   \draw [decorate, decoration={brace, amplitude=8pt}, thick] (1.75,0.8)--(5.75,0.8)
       node [midway, above=3pt]{$j_{n}^{k}+F_{D_{n-k}}$};
   \draw (0,0.4) node {$0$};
   \draw (3.75,-0.4) node {$j_{n}^{k}$};
   \foreach \x in {1.75,5.75} \foreach \y in {0.1,0.2,...,0.7} \draw (\x,\y) -- +(0,0.05);
   \foreach \x in {-2.0,2.0} \foreach \y in {-0.1,-0.2,...,-0.7} \draw (\x,\y) -- +(0,-0.05);
   \foreach \x in {-6.0,6.0} \foreach \y in {-1.3,-1.2,...,-0.1} \draw (\x,\y) -- +(0,-0.05);
\end{tikzpicture}\\
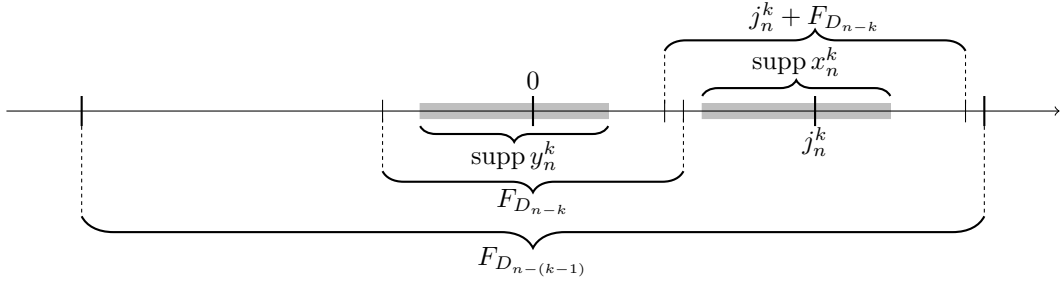
{\scriptsize Figure 1: An illustration of the inclusion $\supp x_{n}^{k}\subset j_{n}^{k}+F_{D_{n-k}}\subset F_{D_{n-(k-1)}}$ in dimension $N=1$.}
\end{center}

\noindent
In particular, for all for $l=0,\ldots,n$ this yields
\begin{equation} \label{eq:left}
\nu(C^{n-l}_n|_{F_{D_l}})
<\nu(B_n)+\delta_l+\delta_{l+1}+\ldots+\delta_n < \nu(B_n)+r_l.
\end{equation}
Finally, we define $C_n:=C_n^n$ and take into account that
$C_n=V_{j_n^n+\ldots+j_n^{n-l+1}} C_n^{n-l} V_{-(j_n^n+\ldots+j_n^{n-l+1})}$,
where $j_n^n+\ldots+j_n^{n-l+1} \in F_{D_l}$ by construction.
%Thus, we get $\nu(C_n|_{F_{3D_l}})\leq\nu(C^{n-l}_n|_{F_{D_l}})< \nu(B_n)+r_l$,
%hence \eqref{EShiftedC}.
Thus, we get
\begin{equation} \label{eq:right}
\nu(C_n|_{F_{3D_l}}) \leq \nu(C^{n-l}_n|_{F_{D_l}}),
\end{equation}
which, together with \eqref{eq:left}, implies \eqref{EShiftedC}.

By the first result in Proposition \ref{PRichOps} we can pass to a subsequence
$(C_{h_n})$ of $(C_n)$ with $\Pc$-strong limit $C\in\sigma_{\op}(A)$.
Then, by sending $n\to\infty$, and since
$(C-C_{h_n})|_{F_{3D_l}}=(C-C_{h_n})P_{3D_l}|_{F_{3D_l}}$ converges to zero
in the norm, we get
\begin{align*}
\nu(C)&\leq\nu(C|_{F_{3D_l}})= \lim_{n\to\infty}\nu(C_{h_n}|_{F_{3D_l}})
\leq \lim_{n\to\infty}\nu(B_{h_n})+r_l
=\inf\{\nu(B):B\in\sigma_{\op}(A)\}+r_l
\end{align*}
for every $l$. Since $r_l\to0$ as $l\to\infty$ we arrive at the assertion.
\end{proof}

\begin{rem}
Here is one (but not the only) way to digest this proof: For every $n\in\Nb$,
we look at $n+1$ translates, $C_n^0,...,C_n^n$, of the operator $B_n$. The lower norm
of these translates is studied on nested sets with decreasing diameters $D_n,...,D_0$.

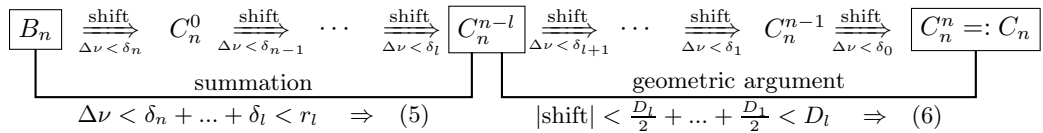
\begin{center}
\begin{tikzpicture}[scale=1.0]
   \draw (0,0) node[draw,rectangle] {$B_n$};
   \draw (2,0) node {$C_n^0$};
   \draw (4,0) node {$\cdots$};
   \draw (6,0) node[draw,rectangle] {$C_n^{n-l}$};
   \draw (8.0,0) node {$\cdots$};
   \draw (10.1,0) node {$C_n^{n-1}$};
   \draw (12.5,0) node[draw,rectangle] {$C_n^n=:C_n$};
   \foreach \x in {1,3,5,9,11} \draw (\x,0.1) node {$\xRightarrow{\rm shift}$};
   \draw (7.1,0.1) node {$\xRightarrow{\rm shift}$};
   \draw (1,-0.25) node {\tiny $\Delta\nu\!<\!\delta_n$};
   \draw (3,-0.25) node {\tiny $\Delta\nu\!<\!\delta_{n-1}$};
   \draw (5,-0.25) node {\tiny $\Delta\nu\!<\!\delta_l$};
   \draw (7.1,-0.25) node {\tiny $\Delta\nu\!<\!\delta_{l+1}$};
   \draw (9,-0.25) node {\tiny $\Delta\nu\!<\!\delta_{1}$};
   \draw(11,-0.25) node {\tiny $\Delta\nu\!<\!\delta_0$};
   \draw [thick] (0,-0.26) -- (0,-0.9) -- (5.8,-0.9) -- (5.8,-0.3);
   \draw [thick] (6.2,-0.3) -- (6.2,-0.9) -- (12.5,-0.9) -- (12.5,-0.28);
   \draw (2.9,-0.7) node {\small summation};
   \draw (2.9,-1.15) node {\small $\Delta\nu<\delta_n+...+\delta_l<r_l\quad \Rightarrow\quad $\eqref{eq:left}};
   \draw (9.35,-0.7) node {\small geometric argument};
   \draw (9.35,-1.15) node {\small $|{\rm shift}|<\frac{D_l}2+...+\frac{D_1}2<D_l\quad \Rightarrow$\quad \eqref{eq:right}};
\end{tikzpicture}
\end{center}

\noindent
For $l\in\{0,...,n\}$, the difference between the restricted lower norm of $C_n^{n-l}$ and the lower norm of $B_n$ is less than $\delta_l+...+\delta_n<r_l$, that is \eqref{eq:left}. On the other hand, the overall shift distance taking $C_n^{n-l}$ to $C_n^n=C_n$ is bounded by $D_l$, which shows \eqref{eq:right}. Together, these two arguments bridge the gap between $\nu(B_n)$ and the restricted lower norm of $C_n$, that is \eqref{EShiftedC}, showing that this gap is bounded by $r_l$, which is independent of $n$. Then first, by sending $n\to\infty$, we pass to the limit of a $\Pc$-convergent subsequence of $C_n$, which is our $C\in\sigma_{\op}(A)$, and finally let $l\to\infty$.
\end{rem}

\begin{cor}\label{CMain}
Let $A\in\Alpr$. If all operators in $\sigma_{\op}(A)$ are invertible then the
norms of their inverses are uniformly bounded.
\end{cor}
\begin{proof}
Assume that the norms of the inverses are not bounded, i.e. there is a sequence
$(B_n)\subset\sigma_{\op}(A)$ with $\|B_n^{-1}\|\to\infty$ as $n\to\infty$. Then
$\nu(B_n)\to 0$ and the previous theorem yields an operator $C\in\sigma_{\op}(A)$
with $\nu(C)=0$, i.e. $C$ is not invertible, a contradiction.
\end{proof}

Now, the desired simplification of Theorem \ref{TUB} is at hand:
\begin{thm}\label{CBig}
Let $A\in\Alpr$, i.e. $A$ is a rich band-dominated operator on a space $\Xb=l^p(\Zb^N,X)$
with parameters $p\in\{0\}\cup[1,\infty]$, $N\in\Nb$ and a Banach space $X$. Then $A$ is
$\Pc$-Fredholm if and only if all its limit operators are invertible.
\end{thm}

Furthermore, we can derive also a formula for the $\Pc$-essential spectrum of $A$, that
is the spectrum of the coset $A+\pc{\Xb}$ in $\pb{\Xb}/\pc{\Xb}$, as it was already
announced in \cite{Marko, LiChW}.
\begin{cor} \label{cor:ess.spec}
Let $A\in\Alpr$. Then
\[\spess(A)=\bigcup_{B\in\sigma_{\op}(A)}\spc(B).\]
\end{cor}
\begin{proof}
If $\lambda\notin\spess(A)$, i.e. $A-\lambda I$ is $\Pc$-Fredholm, then
$\sigma_{\op}(A-\lambda I)$ is uniformly invertible, which yields that $\lambda$
does not belong to the right hand side.
Otherwise, if $A-\lambda I$ is not $\Pc$-Fredholm, then there is a non-invertible
limit operator $B-\lambda I\in\sigma_{\op}(A-\lambda I)$, hence $\lambda$ is
in the set on the right.
\end{proof}

%%%%%%%%%%%%%%%%%%%%%%%%%%%%%%%%%%%%%%%%%%%%%%%%%%%%%%%%%%%%%%%%%%%%%%%%%%%%%%%
%%%%%%%%%%%%%%%%%%%%%%%%%%%%%%%%%%%%%%%%%%%%%%%%%%%%%%%%%%%%%%%%%%%%%%%%%%%%%%%
%%%%%%%%%%%%%%%%%%%%%%%%%%%%%%%%%%%%%%%%%%%%%%%%%%%%%%%%%%%%%%%%%%%%%%%%%%%%%%%
\section{On the necessity of the conditions in Theorem \ref{TMain}}
%%%%%%%%%%%%%%%%%%%%%%%%%%%%%%%%%%%%%%%%%%%%%%%%%%%%%%%%%%%%%%%%%%%%%%%%%%%%%%%
%%%%%%%%%%%%%%%%%%%%%%%%%%%%%%%%%%%%%%%%%%%%%%%%%%%%%%%%%%%%%%%%%%%%%%%%%%%%%%%
%%%%%%%%%%%%%%%%%%%%%%%%%%%%%%%%%%%%%%%%%%%%%%%%%%%%%%%%%%%%%%%%%%%%%%%%%%%%%%%

\begin{ex}
There exists a non-rich diagonal operator $A$ on $\Xb=l^p(\Zb,L^p[0,1])$ for which
the assertion of Theorem \ref{TMain} does not hold. Indeed, define the functions
$b_k(t):=1/k+1+\sin(2\pi kt)$ and the operators $B_k:=b_kI$ on $L^p[0,1]$. Moreover,
let
\[C_n:=\diag(\underbrace{B_1,\ldots,B_1}_{n \text{ times}},
\underbrace{B_2,\ldots,B_2}_{n \text{ times}},
\ldots,\underbrace{B_n,\ldots,B_n}_{n \text{ times}})\]
and, finally, $A:=\diag(\ldots,I,I,C_1,C_2,\ldots,C_n,\ldots)$.
Then $\sigma_{\op}(A)$ consists of the identity $I$, the operators
 $\diag(\ldots,B_k,B_k,\ldots)$, $k\in\Nb$, and all operators
$V_{-l}\diag(\ldots,B_k,B_k,B_{k+1},B_{k+1}\ldots)V_{l}$ with $k\in\Nb$ and $l\in\Zb$.
Clearly, $\inf\{\nu(B):B\in\sigma_{\op}(A)\}=0$, but there is no limit operator
$B\in\sigma_{\op}(A)$ with $\nu(B)=0$.
\end{ex}

\begin{ex}
There exists a rich operator in $\pb{l^p(\Zb,\Cb)}$, $1<p<\infty$, for which the
assertion of Theorem \ref{TMain} does not hold. For its construction we introduce the
$n\times n$-blocks
\[B_n:=\frac{1}{n}\begin{pmatrix}1 & \cdots & 1 \\ \vdots & & \vdots \\ 1 & \cdots & 1\end{pmatrix}\]
and note that $\|B_n\|_p=1$. To see this let $v=(1,\ldots,1)$, observe that
$B_nx=1/n\left\langle x,v\right\rangle v$, and use H\"older's inequality
$\|B_nx\|_p\leq1/n\|x\|_p\|v\|_q\|v\|_p=\|x\|_p$ (where $1/p+1/q=1$) with equality
for $x=v$. Further, $\|B_nP_k\|_p\leq (2k+1)/n\|v\|_p = (2k+1)/n\sqrt[p]{n}\to 0$
as $n\to\infty$.
Now, let $C_k:=I-\frac{k}{k+1}B_k$ for every $k$,
\[D_n:=\diag(\underbrace{C_1,\ldots,C_1}_{n \text{ times}},
\underbrace{C_2,\ldots,C_2}_{n \text{ times}},
\ldots,\underbrace{C_n,\ldots,C_n}_{n \text{ times}})\]
and, finally, again $A:=\diag(\ldots,I,I,D_1,D_2,\ldots,D_n,\ldots)$.
Then $\sigma_{\op}(A)$ consists of the identity $I$, the operators
 $\diag(\ldots,C_k,C_k,\ldots)$, $k\in\Nb$, all operators
$V_{-l}\diag(\ldots,C_k,C_k,C_{k+1},C_{k+1}\ldots)V_{l}$ with $k\in\Nb$ and $l\in\Zb$,
and all operators
$V_{-l}\diag(\ldots,I,I,C_{1},C_{1}\ldots)V_{l}$ with $l\in\Zb$.
Clearly, $\inf\{\nu(B):B\in\sigma_{\op}(A)\}=0$, but there is no limit operator
$B\in\sigma_{\op}(A)$ with $\nu(B)=0$.
\end{ex}

For the extremal $p$ this cannot happen, since \textit{rich} already implies
\textit{band-dominated}, as the next theorem demonstrates.
\begin{thm}
Let $p\in\{0,1,\infty\}$. Then $\pbr{\Xb}=\Alpr$.
\end{thm}
\begin{proof}
As a start, and in view of the characterization of band-dominated operators as given in
\cite[Theorem 2.1.6(e)]{LimOps}, let $A\in\pbr{l^\infty(\Zb^N,X)}$, put
$\varphi_t(s):=e^{\ii(s_{1}t_{1}+\ldots+s_{N}t_{N})}$ for $s\in\Zb^{N}, t\in\Rb^{N}$ and assume that
$A$ is not band-dominated. Then $\|[A,\varphi_tI]\|\not\to 0$ as $t\to 0$.
Thus, there exists a sequence $(t^{(n)})$ in $\Rb^{N}$ tending to zero and a constant $c > 0$
such that $\|[A,\varphi_{t^{(n)}}I]\|\to c$. We choose points $x_n\in l^\infty$,
$\|x_n\|_\infty=1$, such that $\|[A,\varphi_{t^{(n)}}I]x_n\|_\infty\to c$. Now, due to
the definition of  $\|\cdot\|_\infty$ there exists a sequence $(h_n)\subset\Nb$ such
that $\|\chi_{\{0\}}V_{h_n}[A,\varphi_{t^{(n)}}I]x_n\|_\infty\to c$.
Since $A$ is rich and $t^{(n)}\to 0$, we can pass to a subsequence $(g_k)=(h_{n_k})$, such
that the respective $\Pc$-limits of $(V_{g_k}AV_{-g_k})$ and
$(V_{g_k}\varphi_{t^{(n_k)}} V_{-g_k})$ exist. (Note that in the case of $(h_n)$ being
bounded, one can pass to a constant subsequence with the same outcome.)
Clearly, the latter limit is just a multiple of
the identity, let's say $\alpha I$. Nevertheless, from the above we deduce
$\|[A_{g},\alpha I]V_{g_k}x_{n_k}\|_\infty\geq c/2$ for sufficiently
large $k$, a contradiction. So, the equivalence of (a) and (e) in \cite[Theorem 2.1.6]{LimOps}
now provides that $A$ is band-dominated.

Now assume that $p=1$. Then, by the third assertion in Proposition \ref{PRichOps},
$A^*\in\pbr{l^\infty(\Zb^N,X^*)}$, hence $A^*$ is band-dominated by the above
observation. Since $\|[A,\varphi_tI]\|=\|[A,\varphi_tI]^*\|$ for arbitrary $\varphi$,
Theorem 2.1.6 in \cite{LimOps} again provides that $A$ is band-dominated. The case $p=0$ is
analogous.
\end{proof}

\begin{ex}
A final example shall demonstrate, that in the extremal cases there exist very
simple non-rich and non-band-dominated operators in $\pb{\Xb}$ for which the assertion
of Theorem \ref{TMain} does not hold: Consider the flip operator
$J:(x_i)\mapsto(x_{-i})$ and the operator $A=I+J$. Then both $J$ and $A$ have
empty operator spectrum, $J$ is invertible, but $A$ is not $\Pc$-Fredholm.
\end{ex}

%%%%%%%%%%%%%%%%%%%%%%%%%%%%%%%%%%%%%%%%%%%%%%%%%%%%%%%%%%%%%%%%%%%%%%%%%%%%%%%
\subparagraph{Acknowledgements.}
We want to thank Albrecht B\"ottcher, Simon Chandler-Wilde, Vladimir Rabinovich, Steffen Roch and Bernd Silbermann,
whose work on Fredholm properties of infinite matrices has influenced us very much and has left its imprints in many places
of our work and, in particular, of this paper.
We would also like to thank Raffael Hagger from TUHH for helpful discussions and proof-checking.

%%%%%%%%%%%%%%%%%%%%%%%%%%%%%%%%%%%%%%%%%%%%%%%%%%%%%%%%%%%%%%%%%%%%%%%%%%%%%%%
%%%%%%%%%%%%%%%%%%%%%%%%%%%%%%%%%%%%%%%%%%%%%%%%%%%%%%%%%%%%%%%%%%%%%%%%%%%%%%%
%%%%%%%%%%%%%%%%%%%%%%%%%%%%%%%%%%%%%%%%%%%%%%%%%%%%%%%%%%%%%%%%%%%%%%%%%%%%%%%

\end{document}